\documentclass[a4paper,10pt]{article}
\usepackage{amsmath,amsthm,amssymb,amsfonts,epsfig,epstopdf,titling,url,array,enumerate,cite,mathrsfs,upgreek,authblk,scrextend}
\usepackage[bindingoffset=0.2in,left=1in,right=1in,top=1in,bottom=1in,footskip=.25in]{geometry}
\usepackage[bookmarksnumbered, colorlinks, plainpages]{hyperref}
\hypersetup{colorlinks=true,linkcolor=orange, anchorcolor=green, citecolor=cyan, urlcolor=blue, filecolor=magenta, pdftoolbar=true}
\usepackage{doi}

\usepackage{tikz}
\usepackage{pgfplotstable}
\usepackage{blindtext}
\usepackage{pgfplots}
\usepackage{statistics}
\pgfplotsset{compat=1.16}
\usepgfplotslibrary{external}
\usepackage{graphicx}
\graphicspath{ { ./images/ } }
\usepackage{wrapfig}
\usepackage{float}
\theoremstyle{plain}
\newtheorem{thm}{Theorem}[section]

\newtheorem{lma}[thm]{Lemma}

\newtheorem{ppn}[thm]{Proposition}
\theoremstyle{definition}
\newtheorem{dfn}[thm]{Definition}
\newtheorem{eg}[thm]{Example}

\providecommand{\ams}[1]{\begin{addmargin}[28pt]{28pt}\noindent\textbf{Mathematics Subject Classification:} #1\end{addmargin}}

\title{Fixed point theorems on perturbed metric space with an application}
\author[1]{Dipti Barman}
\author[2]{T. Bag \thanks{Corresponding author}}
\affil[1,2]{Department of Mathematics, Siksha-Bhavana,\authorcr
	Visva-Bharati, Santiniketan, Birbhum, West-Bengal, India-731235\authorcr
     0009-0006-1146-4566$^1$, 0000-0002-8834-7097$^2$\authorcr
	E-mail\textsuperscript{1}: diptibarmanhmt@gmail.com\authorcr
	E-mail\textsuperscript{2}: tarapadavb@gmail.com}
\date{}

\begin{document}
    	\maketitle
    	 \pagestyle{myheadings}
\markright{\footnotesize \it   Fixed point theorems on perturbed metric space with application} 
    	\begin{abstract}
Following the definition of perturbed metric space, in this paper, some fixed point theorems are established for $ F $-perturbed mappings in complete perturbed metric spaces and justify the result by  counter example.  Finally, an application of this theorem for the existence of a solution for the second-order boundary value problem is given.	
    	\end{abstract}

    \ams{46T99, 47H10, 54H25}
         \section{Introduction}


 In modern mathematics, metric spaces and normed linear spaces are two widely used concepts in functional analysis. The notion of metric was developed by Frechet and later Hausdorff presented it axiomatically. After that, several authors developed the metric idea in different approaches. They tried to develop functions which are more general than metric function. In 1963, Gahler \cite{2-metric}  introduced $ 2 $-metric and later $ n $-metric and established many fundamental results of functional analysis. $ D $-metric \cite{D-metric}, $ S $-metric \cite{s-metric},  $ b $-metric \cite{b-metric}, $ G $-metric \cite{G-metric}, $ F $-metric \cite{F-metric} etc. are some generalized metrics defined by different authors. In generalized metric, authors mainly modify the triangle inequality in the Definition of basic metric space. It is worth mentioning the contribution of several authors on fixed point results in generalized metric spaces. For references please see \cite{Generalized theta parametric,Fixed-point-1, Fixed-point-2, Fixed-point-3,Fixed-point-4,Fixed-point-5, Fixed-point-6}. \\

 The measurement of the distance between two points is always tainted by errors. For instance, the imperfection in the adjustment of instruments affects the accuracy of measurements. These errors are small, however their accumulations can be significant. Jleli  et. al \cite{P-metric} realized this fact and introduced the notation of perturbed metric space. They also established some fixed-point theorems in such space. We study  the literature of perturbed metric space and realize that there are huge scope to develop fixed point results for various type of contraction mappings in perturbed metric spaces. \\

 Motivated by this fact, in this paper we introduce the $ F $-perturbed mapping and establish fixed point theorem in perturbed metric space with application for the existence of  solution for a second order boundary valued problem. The effectiveness of this approach is illustrated by a numerical experiment.


    	   \section{Preliminaries}
    	   \begin{dfn}
 \cite{P-metric}   	   	Let $ D, P : X \times X \to [0, \infty) $ be two given mappings. $ D $ is a perturbed metric on $ X $ with respect to $ P $ if the function 
    	   	$$ d = D - P : X \times X \to \mathbb{ R }, ~ (x, y) \to D (x, y) - P (x, y) $$
    	   	is a metric on $ X $. This means that for all $ x, y,  z \in X $, the following conditions hold: 
    	   	\begin{enumerate}[(P1)]
    	   		\item $ ( D - P ) (x, y) \geq 0; $
    	   		\item $ ( D - P ) (x, y) = 0 \iff x = y; $
    	   		\item $ ( D - P ) ( x, y ) = ( D - P) (y, x); $
    	   		\item $ ( D - P ) ( x, y) \leq ( D - P) ( x, z) + ( D - P) (z, y). $
    	   	\end{enumerate}
    	   	We call $ P $ a perturbed mapping, $ d = D - P $ the exact metric, and $ ( X, D, P) $ a perturbed metric space.
    	   \end{dfn}

          \begin{eg}\label{perturbed example 1 }\cite{P-metric}
     Let $ D : \mathbb{ R } \times \mathbb{ R } \to [0, \infty) $ be the mapping defined by 
           $$ D (x, y) = | x -y | + x^2y^4, \quad \forall~ x, y \in \mathbb{ R }. $$
           Then $ D $ is a perturbed metric on $ \mathbb{ R } $ with respect to the perturbed mapping 
            $$ P : \mathbb{ R } \times \mathbb{ R } \to  [0, \infty) $$ given by 
           $$ P (x, y) = x^2y^4, \quad  \forall ~ x, y \in \mathbb{ R }. $$
           In this case, the exact metric is the mapping $ d  $ defined by 
           $$ d (x, y) = | x - y |, \quad \forall~ x, y \in \mathbb{ R }. $$

       \end{eg} 
           
           \begin{figure}
           [ht!]
               \centering
               \includegraphics[width=20cm, height=10cm]{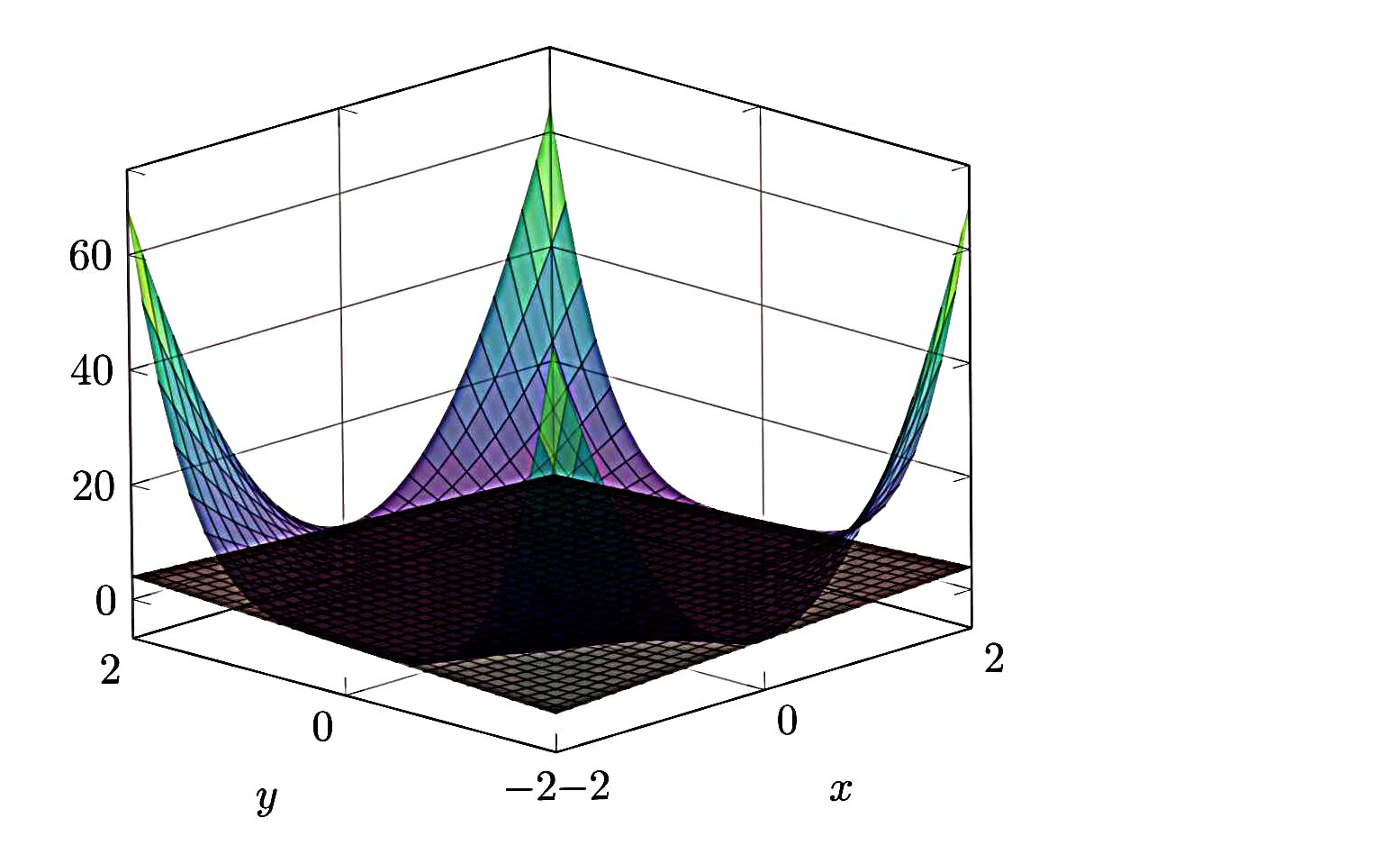}
               \caption{Curve type picture indicates the value of $ D(x, y) = |x - y| + x^2y^4  $ where as rectangular type picture indicates the value of $ d(x, y) = |x - y| $ }
           \end{figure}

            %
            %
             %
           %
           %
          We present the following basic properties of a perturbed metric space, as established by Jleli and Samet \cite{P-metric}. 
          \begin{ppn}
              Let $ D, P, Q : X \times X  \to  [0, \infty) $ be three given mappings and  $ \alpha > 0 $. 
              \begin{enumerate}[(i)]
                  \item If  two triples $ ( X, D, P) $ and $ (X, D, Q ) $ denote two perturbed metric spaces, then the triple $ ( X, D, \frac{ P+Q}{ 2 } ) $ forms a perturbed metric space also.
                  \item The triple $ ( X, \alpha D, \alpha P) $ forms a perturbed metric space if the triple $ (X, D, P) $ consists of a perturbed metric space. 

              \end{enumerate}
          \end{ppn}

            Jleli and Samat \cite{P-metric} introduced the following topological notations related to perturbed metric spaces $ (X, D, P) $.
            
            \begin{dfn}\label{Topological notations}
            	Let $ (X, D, P) $ be a perturbed metric space. Let $ \{ x_n \} $ be  a sequence in $ X $, and $ T : X \to X $ be a mapping.
            	
            	\begin{enumerate}[(i)]
            		\item We say that $ \{ x_n \} $ is a perturbed convergent sequence in $ ( X, D, P ) $ if $ \{ x_n \} $ is convergent in the metric space $ (X, d) $, where $ d = D - P $ is the exact metric.  
            		
            		\item We say that $ \{ x_n \} $ is a perturbed Cauchy sequence in $ ( X, D, P ) $ if $ \{ x_n \} $ is a Cauchy sequence in the metric space $ ( X, d ) $.
            		
            		\item We say that $ ( X, D, P) $ is a complete perturbed metric space if $ (X, d) $ is a complete metric space, or equivalently, if every perturbed Cauchy sequence in $ ( X, D, P ) $ is a perturbed convergent sequence in $ (X, D, P) $.
            		
            		\item We say that $ T $ is a perturbed continuous mapping if $ T $ is continuous with respect to the exact metric $ d $.
            	\end{enumerate}
            \end{dfn}

We develop a fixed-point theorem in Section 3 using the following type of functions.

    	  \begin{dfn}\cite{7}
    	  	Let $ F : \mathbb{R_+} \to \mathbb{ R } $ be a function satisfying
    	  	\begin{enumerate}[(F1)]
    	  		\item F is strictly increasing,  i.e., for all $ t_1, t_2 \in \mathbb{R_+} $ such that $ t_1 < t_2 $  $ \implies  F ( t_1 ) < F ( t_2 ) $,
    	  		\item For each sequence $ \{ t_n \}_{ n \in \mathbb{N} } $ of positive real numbers,  $ \underset{ n \to \infty }{ \lim } t_n = 0 \iff \underset{ n \to \infty }{ \lim } F ( t_n ) = - \infty $,
    	  		\item There exists $ K \in (0, 1) $ such that $ \underset{ t \to 0^+ }{ \lim  } t^k F ( t ) = 0 $.
    	  	\end{enumerate}
    	  \end{dfn} 
Some examples of $ F : R_+ \to R $  mappings which satisfy the conditions $ (F1), (F2) $, and $ ( F3) $ as follows.

\begin{eg}
    \begin{enumerate}[(i)]
        \item $ F (x) = \ln x, \quad x \in (0, \infty) $.
        \item $ F (x) = \ln x +x, \quad x \in (0, \infty) $.
        \item $ F (x) = - \frac{ 1 }{ \sqrt{ x } }, \quad x \in (0, \infty) $.
        \item $ F (x ) = \ln ( x^2 + x), \quad  x \in (0, \infty) $.
    \end{enumerate}
\end{eg}

 We recall the $ F $-contraction theorem from \cite{7} in  exact metric space 
for use in the main result.

\begin{dfn}\cite{7}
    A mapping $ T : X \to X $ is said to be an $ F $-contraction if there exists $ \tau > 0 $ such that 
    $$ d (Tx, Ty) > 0, \quad \forall~ x, y  \in X  \implies \tau + F ( d (Tx, Ty)) \leq F (  d (x, y) ), \quad  \forall~ x, y \in X. $$
\end{dfn}

\begin{thm}\cite{7}
    Let $ (X, d) $ be a complete exact metric space and let $ T : X \to X $ be a $ F $-contraction. $ T $ has a unique fixed point $ x^* \in X $, and for every $ x_0 \in X  $ a sequence $ \{ T^nx_0  \}_{ n \in \mathbb{ N} } $ is convergent to $ x^* $. 
\end{thm}

    	  \section{ \textbf{Main results} }
    	  Inspired by the work of D. Wardowski \cite{7} in metric spaces, we now extend our results in a setting of a perturbed metric space.
    	Now we begin this section by introducing the concept of $ F $-perturbed mapping and a lemma which is important in proving the main results.
    	
            \begin{lma}\label{lemma 1}
             Let $ X $ be a nonempty set and $ D (x, y ) $ be a perturbed metric on $ X $ with respect to the perturbed mapping  $ P (x, y) \geq 0, \quad \forall ~ x, y \in X $. Define 
             $$ D (x, y) = d (x, y) + P (x, y). $$
             If $ x = y $ then $ D (x, y) = P (x, y) \geq 0 $. \\
             If $ x \ne y $ then $ d (x, y) > 0 $. So that $ D (x, y) > 0 $.

        \end{lma}
    	\begin{dfn}
    	A mapping $ T : X \to X $ is said to be an $ F $-perturbed mapping if there exists $ \tau > 0 $ such that 
    	\begin{equation}\label{ Perturbed eq 1}
    	 D ( Tx, Ty ) > 0 \implies \tau + F ( D ( Tx, Ty ) ) \leq F ( D (x, y) ), \quad \forall~ x, y \in X. 	
    	\end{equation}
    \end{dfn}

    \begin{thm}\label{Perturbed fixed point theorem}
        Let $ (X, D, P) $ be a complete perturbed metric space and $ T: X  \to X $ be a perturbed continuous $ F $-perturbed mapping. Then $ T $ has a unique fixed point in $ X $.
    \end{thm}
          
          \begin{proof}
          	Let $ x_0 \in X $ be arbitrary and fixed. We define a sequence $ \{ x_n \} \subset X $, $ x_{ n +1 } = T x_n, ~ n = 1, 2, 3, \ldots $ and denote $ \gamma_n = D ( x_ { n +1 }, x_n ), \quad  n = 1, 2, 3, \ldots $. \\
          	If there exists $ n_0 \in \mathbb{N} $ such that $ x_{ n_0 + 1 } = x_{ n_0 } $ then $ Tx_{ n_0 } = x_{ n_0 } $ and the proof is done.\\
          	Suppose $ x_{ n + 1 } \ne x_n,  \quad  \forall ~ n \in \mathbb{N} $. Then $ \gamma_n > 0, \quad  \forall ~  n  \in \mathbb{ N } $  (using Lemma \ref{lemma 1}). \\
          	Using (\ref{ Perturbed eq 1}), the following holds for every $ n \in \mathbb{N} $ 
          	\begin{align}\label{Perturbed eq 2}
          	& F ( D ( x_{ n+ 1}, x_n ) ) + \tau \leq F ( D ( x_n, x_{ n-1} ) ) \nonumber \\
          	 \implies & F ( D (x_{n +1}, x_n) ) \leq F ( D ( x_n, x_{n-1} ) ) - \tau \nonumber \\
          	 \implies & F ( \gamma_n ) \leq F ( \gamma_{ n-1} ) - \tau. 	
          	\end{align}
           Again from (\ref{ Perturbed eq 1}), we have 
           \begin{align*}
           &	F ( \gamma_{ n-1} ) + \tau \leq F ( \gamma_{ n- 2} ) \\
           	\implies & F ( \gamma_{n-1} )  \leq F ( \gamma_{ n- 2}) - \tau.
           \end{align*}
           From (\ref{Perturbed eq 2}), we have 
           \begin{align}\label{Perturbed eq 3}
           &	F ( \gamma_n ) \leq F ( \gamma_{ n- 2} ) - \tau - \tau \nonumber \\
           \implies & F ( \gamma_n ) \leq F ( \gamma_{ n- 2} ) - 2 \tau \nonumber  \\
           &                          \vdots \nonumber \\
          \implies  &  F ( \gamma_n ) \leq F( \gamma_0 ) - n \tau, \quad n \in \mathbb{N}	 \\
           \implies & \underset{  n \to \infty }{ \lim } F ( \gamma_n ) \leq \underset{ n \to \infty }{ \lim } [ F ( \gamma_0 ) - n \tau ] \leq - \infty , \quad n \in \mathbb{ N } \nonumber \\
           \implies & \underset{ n \to \infty }{ \lim } F ( \gamma_n ) = - \infty. \nonumber 
           \end{align}
          By $ (F2) $, we have 
           \begin{equation}\label{Perturbed eq 4}
           	\underset{ n \to \infty }{ \lim } \gamma_n = 0. 
           \end{equation}
       By $ (F3) $,  there exists $ k \in (0, 1) $ such that 
            \begin{equation}\label{Perturbed eq 5}
              \underset{ n \to \infty }{ \lim } \gamma_n^k F ( \gamma_n ) = 0.	
            \end{equation}
            From (\ref{Perturbed eq 3}), we have
            \begin{align*}
            	& F ( \gamma_n ) \leq F ( \gamma_0 ) - n \tau \\
            	\implies & \gamma_n^k F ( \gamma_n ) \leq  \gamma_n^k F ( \gamma_0 ) - n \tau \gamma_n^k \\
            	\implies & \gamma_n^k F ( \gamma_n ) - \gamma_n^k F ( \gamma_0 ) \leq - n \tau \gamma_n^k \\
            	\implies & \underset{ n \to \infty }{ \lim } [ \gamma_n^k F ( \gamma_n ) - \gamma_n^k F ( \gamma_0 ) ] \leq \underset{ n \to \infty }{ \lim } [ -n \tau \gamma_n^k ]  \\
            	\implies & \underset{ n \to \infty }{ \lim } \gamma_n^k F ( \gamma_n ) - \underset{ n \to \infty }{ \lim } \gamma_n^k F ( \gamma_0 ) \leq \underset{ n \to \infty }{ \lim } [ - n \tau \gamma_n^k  ] \\
            	\implies & 0 \leq \underset{ n \to \infty }{ \lim } [ -n \tau \gamma_n^k ] \leq 0 \\
            	\implies & \underset{ n \to \infty }{ \lim } n \tau \gamma_n^k = 0 \\
            	\implies & \underset{ n \to \infty }{ \lim } n \gamma_n^k = 0.  
            \end{align*}
i.e.,     ~    for $ \epsilon > 0 $ there exists $ N \in \mathbb{ N } $ such that 
                \begin{align*}
                	& n \gamma_n^k < \epsilon, \quad  \forall ~n \geq N \\
                	\implies & \gamma_n^k < \frac{ \epsilon }{ n }, \quad \forall~  n \geq N \\
                	\implies & \gamma_n < ( \frac{ \epsilon }{ n } )^{ \frac{ 1 }{ k } }, \quad  \forall~ n \geq N  \\
                	\implies & D ( x_{ n +1 }, x_n ) < ( \frac{ \epsilon }{ n } )^{ \frac{ 1 }{  k } }, \quad  \forall ~  n \geq N \\
                	\implies & d ( x_{ n +1 }, x_n ) + P ( x_{ n + 1 }, x_n ) <  ( \frac{ \epsilon }{ n } )^{ \frac{ 1 }{ k } }, \quad \forall ~ n \geq N \\
                	\implies & d ( x_{ n + 1}, x_n ) < ( \frac{ \epsilon }{ n } )^{ \frac{ 1 }{ k } },
                	 \quad  \forall ~ n \geq N .
                \end{align*}
            Now we will show that $ \{ x_n \} $ is a Cauchy sequence. 
            \begin{align*}
            	d ( x_{ n +p }, x_n ) & \leq d ( x_{ n + p }, x_{ n + p -1 } ) + d ( x_{ n+p-1}, x_{n+p-2} ) + \ldots +  d (x_{ n+1}, x_n ) \\
            	& \leq \frac{ \epsilon^{ \frac{ 1 }{ k } } }{ ( n +p - 1 )^{\frac{ 1 }{ k } } } + \frac{ \epsilon^{ \frac{ 1 }{ k } } }{ ( n + p - 2)^{ \frac{ 1 }{ k } } } + \ldots +  \frac{ \epsilon^{ \frac{ 1 }{ k } } }{ n^{ \frac{ 1 }{ k } } }, \quad  \forall~ n \geq N  \\
            	& \leq \epsilon^{ \frac{ 1 }{ k } } \bigg[ \frac{ 1 }{ ( n + p -1)^{ \frac{ 1 }{ k } } } + \frac{ 1 }{ ( n + p -2 )^{ \frac{ 1 }{ k } } } + \ldots + \frac{ 1 }{ n^{ \frac{ 1 }{ k } } } \bigg], \quad  \forall~  n \geq N \\
            	& = \epsilon^{ \frac{ 1 }{ k } } \sum_{ i = 0 }^{ p -1 } \frac{ 1 }{ ( n + i)^{ \frac{ 1 }{ k } } }, \quad  \forall ~  n \geq N.
            \end{align*}
        Since $ \sum_{ i = 0 }^{ p -1 } \frac{ 1 }{ ( n + i)^{ \frac{ 1 }{ k } } } $ is  convergent and converges to zero, so $ \underset{ n \to \infty }{ \lim } d (x_{ n + p }, x_n ) = 0, \quad  p = 1, 2, 3, \ldots $. \\
        Hence, $ \{ x_n \} $ is a Cauchy sequence in the exact metric $ ( X, d ) $ that is $ \{ x_n \} $ is a perturbed Cauchy sequence in the perturbed metric space $ ( X, D, P) $. By the completeness of the perturbed metric space $ (X, D, P) $, we deduce that there exists $ x^* \in X $ such that 
        $$ \underset{ n \to \infty }{ \lim } d (x_n, x^* ) = 0. $$ 
        We now show that $ x^* $ is a fixed point of $ T $. Since $ T $ is a perturbed continuous mapping, then 
        $$ \underset{ n \to \infty }{ \lim } d ( Tx_n, Tx^* )= 0 \quad  \text{ i.e., } \quad   \underset{ n \to \infty }{ \lim } d ( x_{ n + 1 }, Tx^* ) = 0. $$
        From the triangle inequality of $ d $, we have 
        \begin{align*}
        &	 d ( x^*, Tx^* ) \leq d ( x^*, Tx_n ) + d ( Tx_n, Tx^* ) 
        	  =  d (x^*, x_{ n + 1 } ) + d ( x_{ n + 1 }, Tx^* ) \\
        	\implies  &   \underset{ n \to \infty }{ \lim } d ( x^*, Tx^* )   \leq \underset{ n \to \infty }{ \lim } [ d (x^*, x_{ n + 1 } ) + d ( x_{ n+1 }, Tx^* ) ]  \\
        	\implies & d ( x^*, Tx^*) = 0 \\
        	\implies & Tx^* = x^*.
        \end{align*}
    To prove that $ x^* $ is a unique fixed point.\\
    Suppose there exists $ y^* \in X $ such that $ Tx^* = x^* \ne y^* = Ty^* $.
    \begin{align*}
     & \tau + F ( D ( Tx^*, Ty^* ) ) \leq F ( D (x^*, y^* ) ) \\
     \implies & \tau + F ( D ( x^*, y^* ) ) \leq F ( D ( x^*, y^* ) )  \\
     \implies & \tau \leq  F ( D (x^*, y^*) ) - F ( D  ( x^*, y^*) ) \\
     \implies & \tau \leq 0, ~ \text{ a contradiction. }	
    \end{align*}
               Hence $ x^* $ is a unique fixed point. The proof is completed.
          \end{proof} 	

     We now present an example which supports  the  Theorem  \ref{Perturbed fixed point theorem}.      
            \begin{eg}\label{fixed point example 1}
               Let $ D : [0, 1] \times [0, 1] \to [0, \infty) $ be the mapping defined by 
               $$ D (x, y) = | x- y| + (x-y)^4, \quad  \forall ~  x, y \in [0, 1]. $$
               Then $ D $ is a perturbed metric on $ [0, 1] $ with respect to the perturbed mapping $ P : [0, 1] \times [0, 1] \to [0, \infty) $ given by $ P (x, y) = ( x - y)^4 $.\\
               In this case, the exact metric is the mapping $ d : [0, 1] \times [0, 1] \to [0, \infty) $ defined by 
               $$ d (x, y)= | x - y|, \quad  \forall~ x, y \in [0, 1]. $$	
For $ x = 0, y = \frac{ 1 }{ 2 }, z = \frac{ 1 }{ 3 } $, we have
               \begin{align*}
               	D (0, \frac{ 1 }{ 2 }) = 0.5625, ~ D (0, \frac{ 1 }{ 3}) + D ( \frac{ 1 }{ 3 }, \frac{ 1 }{ 2 } ) = 0.513
               \end{align*}
            which  shows that $$ D (0, \frac{ 1 }{ 2 }) > D (0, \frac{ 1 }{ 3 }) + D ( \frac{ 1 }{ 3 }, \frac{ 1 }{ 2} ). $$
            Hence, $ D $ is not an exact metric on $ [0, 1] $. \\
            Let us define a mapping $ T : [0, 1] \to [0, 1] $ by 
            $$ Tx = \frac{ x }{ 2 }, \quad  x \in [0, 1]. $$
            Take $ F (x) = \log x, \quad x \in (0, \infty)  $, which satisfies the conditions $ (F1), (F2), (F3) $ and $ \tau = \log2 > 0 $.\\
            Then we have
            \begin{align*}
            	D (Tx, Ty) & = |Tx - Ty| + ( Tx - Ty)^4  \\
            	& = \bigg| \frac{ x }{ 2 } - \frac{ y }{ 2 } \bigg| + ( \frac{ x }{ 2 } - \frac{ y }{ 2} )^4 \\
            	& = \frac{ | x - y|}{ 2 } + \frac{ (x-y)^4 }{ 16 }.
            \end{align*}
            Clearly, we have 
            \begin{align*}
            	\tau + F ( D ( Tx, Ty) ) & = \log2 + \log ( D ( Tx, Ty) ) \\
            	& = \log ( 2 D ( Tx, Ty) ) \\
            	& = \log \Big[ 2 \Big \{ \frac{ | x - y|}{ 2 } + \frac{ (x - y)^4 }{ 16 } \Big \} \Big] \\
            	& = \log \Big [ | x - y | + \frac{ ( x - y)^4 }{ 8 } \Big ] \\
            	& \leq \log \Big [ | x -y | + ( x - y)^4 \Big ]  \\
            	& = F ( D (x, y) ).
            \end{align*}
        Hence, $ T $ is a $ F $-perturbed mapping. So, by the Theorem \ref{Perturbed fixed point theorem},  $ T $ has a unique fixed point in $ X $ which is $ x = 0 $.
            \end{eg}
        
            \begin{figure}
            
                \centering
                \includegraphics[width=20cm, height=10cm]{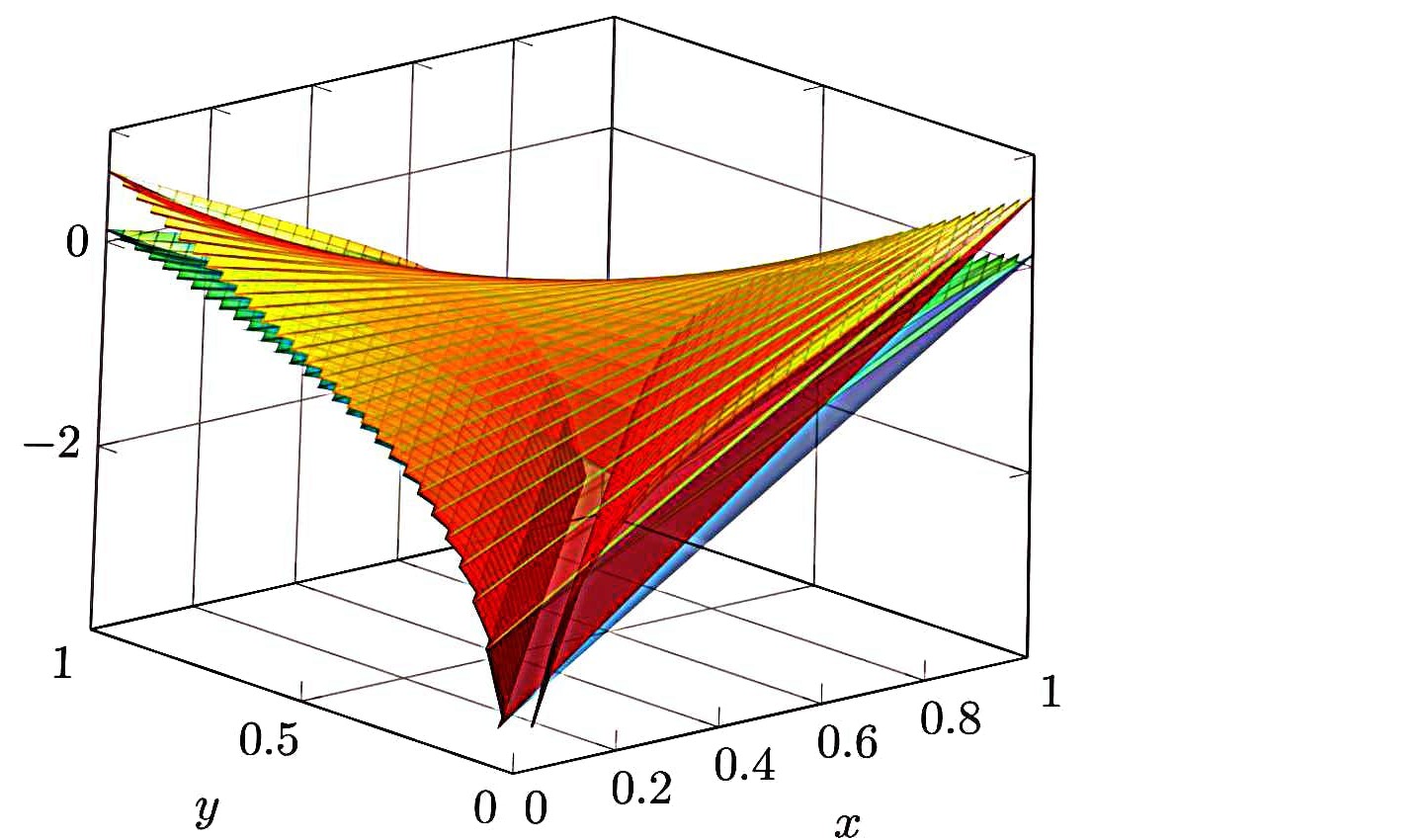}
                \caption{The upper picture indicates $ F( D(x, y)) $ where, as the lower picture indicates $ \tau + F ( D (Tx, Ty)) $}
                \label{fig:enter-label}
            \end{figure}

\vspace{0.5cm}

            \section{\textbf{Existence of a Solution to the Boundary Value Problem}}
              We give an application for the second-order boundary value problem of a unique solution with a binary relation, which is applicable in the Theorem  \ref{Perturbed fixed point theorem}. \\
              Let us consider a second-order boundary value problem as follows
              \begin{equation}\label{Boundary value problem 1}
                - u^{\prime \prime } ( t ) = f ( t, u (t) ), \quad  t \in (0, 1) \quad 
                u(0) = u(1) = 0	
              \end{equation}
          where the mapping $ f : (0, 1) \times \mathbb{ R } \to \mathbb{ R } $ is continuous. \\
          This type of problem arises in Physics and Engineering, such as in heat conduction, elastic deformation, and electrostatics.\\
          One can easily show that the problem (\ref{Boundary value problem 1}) is equivalent to the integral equation
          $$ u ( t ) = \int_{ 0 }^{ 1 } G (t, s) ~f (s, u(s) )~  ds $$
          where $ G (t, s) $ is Green's function given by
          $$ G (t, s) = \begin{cases}
          	s ( 1 - t), \quad  0 < s \leq t \leq 1 \\
          	t ( 1 - s ), \quad  0 < t < s \leq  1.
          \end{cases} $$
          The Green's function satisfies the boundary conditions $ u ( 0 ) = u ( 1 ) = 0 $. \\
          Let $ X = C[0, 1] $, the class of all real-valued continuous functions defined on $ [0, 1] $ and define a function
          \begin{equation}\label{Boudary value problem 2}
          	 D (u_1(t), u_2(t)) = \underset{ t \in [0, 1] }{ \sup } | u_1 ( t ) - u_2 ( t ) | +  | u_1 ( 0 ) - u_2 ( 0 ) |, \quad   \forall~ u_1, u_2 \in X.
          \end{equation} 
           $ D $ is a perturbed  metric on $ C[0, 1] $ with respect to the perturbed mapping $ P : C [0, 1] \times C[0, 1] \to [0, \infty) $ defined by the relation 
           $$ P ( u_1 (t), u_2 (t) ) = | u_1 (0) - u_2 ( 0 ) |, \quad  \forall ~ u_1, u_2 \in C[0, 1]. $$
           The exact metric is the function $ d : C [0, 1] \times C[0, 1] \to [0, \infty ) $ defined by the relation 
          \begin{equation}\label{Boundary value problem 3}
          	d (u_1 (t), u_2 (t) ) = \underset{ t \in [0, 1] }{ \sup } | u_1 ( t ) - u_2 ( t ) |, \quad \forall ~   u_1, u_2 \in C[0, 1].
          \end{equation}
           It is known that the set $ C[0, 1] $ endowed with the metric $ d $ defined by the relation (\ref{Boundary value problem 3}), that is $ ( C[0, 1], d ) $ is a complete metric space. Consequently, by (iii) of the definition (\ref{Topological notations}), it is found that $ ( C[0, 1], D, P) $ is a complete perturbed metric space. 
           
           \begin{thm}
           	Consider the perturbed metric space $ (X, D, P) $ defined on (\ref{Boudary value problem 2}). Suppose that the boundary value problem satisfies the following condition
           	\begin{equation}\label{Boundary value problem 4}
           	 | f \big(s, u_1 ( s) \big) - f \big(s, u_2 ( s ) \big) | \leq \exp({ - \tau })  \cdot | u_1 ( s ) - u_2 ( s )  |, \quad \tau > 0.	
           	\end{equation}
           Then the boundary value problem (\ref{Boundary value problem 1}) has a unique solution.
           
           \begin{proof}
           	 We define a mapping $ T : X \to X $ by 
           	 $$ T \big( u (t ) \big)  =  \int_{ 0 }^{ 1 } G (t, s) ~ f \big(s, u ( s ) \big) ~  ds, \quad  s \in [0, 1]. $$
           	 For all $ u_1, u_2 \in X $, we have 
           	 \begin{align*}
           	 	\tau + \ln  \big(D ( Tu_1, Tu_2 ) \big)  & = \tau + \ln ~ [ \underset{ t \in [0, 1] }{ \sup } ~ | Tu_1 ( t ) -Tu_2 ( t ) | + | Tu_1 ( 0 ) - Tu_2 ( 0 ) |  ~]   \\
           	 	 & = \tau  + \ln ~ \Big[ ~\underset{ t \in [0, 1] }{ \sup } ~ \Big|  \int_{ 0 }^{ 1 } G (t, s) ~ f \big(s, u_1 ( s ) \big) - \int_{ 0 }^{ 1 } G (t, s)~ f \big(s, u_2 ( s ) \big)  \Big| \\
           	 	 &  + 
           	 	   \Big| \int_{ 0 }^{ 1 } G (0, s)~ f \big( s, u_1 ( s ) \big) - \int_{ 0 }^{ 1 } G (0, s)~ f \big(s, u_2 ( s ) \big) \Big| ~~\Big] \\
           	 	 &  \leq \tau + \ln ~ \Big[ ~  \underset{ t \in [0, 1] }{ \sup } ~ \Big| \int_{ 0 }^{ 1 } g (t, s)~ ds | \cdot | \int_{ 0 }^{ 1 } \big( f (s, u_1 ( s ) ) - f ( s, u_2 ( s ) ) \big)~ ds \Big| \\
           	 	 &  +  \underset{ t \in [0, 1] }{ \sup } ~ \Big| \int_{ 0 }^{ 1 } g (0, s) ~ ds \Big| \cdot \Big| \int_{ 0 }^{ 1 } \big( f (s, u_1 ( s ) ) - f ( s, u_2 ( s ) ) \big) ~  ds \Big| ~~\Big]  \\
           	 	 &  < \tau + \ln ~ \Big[ ~  \frac{ 1 }{ 2 }  \cdot  \Big| \int_{ 0 }^{ 1 } \big( f (s, u_1 ( s ) ) - f ( s, u_2 ( s ) ) \big) ~ ds \Big|  \\
           	 	 &  + \frac{ 1 }{ 2 } \cdot  \Big| \int_{ 0 }^{ 1 } \big( f (s, u_1 ( s ) ) - f ( s, u_2 ( s ) ) \big)~ ds \Big|~~ \Big]  \\
           	 	 & < \tau +  \ln~  \Big[~  \Big| \int_{ 0 }^{ 1 } \big( f (s, u_1 ( s ) ) - f ( s, u_2 ( s ) ) \big) ~ ds \Big|~~ \Big]  \\
           	 	 & = \tau + \ln ~ \Big[ \underset{ s \in [0, 1] }{ \sup }~ \big| f (s, u_1 ( s ) ) - f (s, u_2 ( s ) )  \int_{ 0 }^{ 1 } ~ ds \big|~~ \Big]  \\
           	 	 &  = \tau + \ln~ [~ \exp({ - \tau } ) \cdot \underset{ s \in [0, 1]}{ \sup } ~ | u_1 ( s )  - u_2 ( s ) | ~~]  \\
           	 	 & = \tau + \ln~ [~ \exp({ - \tau } ) ~] + \ln [ ~\underset{ s \in [0, 1]}{ \sup } ~  | u_1 ( s ) - u_2 ( s ) |~ ]   \\
           	 	 & = \tau - \tau +  \ln~ [~ \underset{ s \in [0, 1]}{ \sup } ~| u_1 ( s ) - u_2 ( s ) |~ ]  \\
           	 	 & = \ln ~[ ~ \underset{ t \in [0, 1] }{ \sup } ~| u_1 ( t ) - u_2 ( t ) | +  | u_1 ( 0 ) - u_2 ( 0 ) | ~]  = \ln~ [~ D \big(  u_1 ( t ), u_2 ( t ) \big) ~] 
           	 \end{align*}
            $ \implies 	\tau + \ln~ [~ \big( D ( Tu_1, Tu_2 ) \big) ~]  < \ln ~[~ D \big( u_1 ( t ), u_2 ( t ) \big) ~]  $. \\
            Thus, the mapping $ T $ fulfills the conditions of the Theorem  \ref{Perturbed fixed point theorem} and therefore $ T $ has a unique fixed point in $ X $.  Consequently, the boundary value problem (\ref{Boundary value problem 1}) has a unique solution in $ X $.
           \end{proof} 
          
\section{Numerical Example} \end{thm}
In this section, a numerical example is established to indicate the significance of the given results. Let $ X $ be a set of all continuous real-valued functions defined on $ [0, 1] $, i.e, $ X = C [0, 1] $ and define $ D : X \times X  \to  [0, \infty) $ by 
$$ D  \big(u_1(t), u_2(t) \big) =  \underset{ t \in [0, 1] }{ \sup } ~ |u_1 (t) - u_2 (t) | + |u_1(0) - u_2(0) |, \quad \forall ~ u_1, u_2 \in X.  $$
Then $ D $ is a perturbed metric on $ X $ with respect to the perturbed mapping $ P : X \times  X \to [0, \infty) $ defined by the relation
$$ P \big( u_1 (t ), u_2(t) \big) = |u_1( 0 ) - u_2(0) |, \quad \forall~  u_1, u_2 \in X. $$
and the exact metric is the function $ d : X \times  X  \to [0,\infty) $ defined by $$ d \big ( u_1(t), u_2(t) \big) = \underset{ t \in [0, 1] }{ \sup } ~ |u_1(t) - u_2(t) |, \quad \forall ~ u_1, u_2 \in X. $$
Clearly, $ (X, D, P) $ is  complete perturbed metric space.\\
Let $ T $ be the operator defined by 
\begin{equation}\label{Nu eq -1} 
T u(t) = \int_0^1 G(t, s)~ f \big( s, u(s) \big)~ ds,\quad s \in [0, 1]
\end{equation}
where $$  G (t, s) = \begin{cases}
    s (1-t), \quad  0 < s \leq t \\
    t (1-s), \quad t \leq s < 1.
\end{cases}$$
Let $ f \big(s, u(s) \big) = \big(  \frac{ s + 0.5 }{ 2 } \big) \cdot \sin u(s) $. Then (\ref{Nu eq -1}) becomes
\begin{equation}\label{Nu eq-2}
    T u (t) = \int_0^1 G (t, s) ~ \big( \frac{ s + 0.5 }{ 2 }\big) \cdot \sin u (s)~ ds, \quad u (t) \in X. 
\end{equation}
Suppose the following condition holds
$$  \big| f \big( s, u_1(s) \big) - f  \big(  s, u_2( s ) \big) \big| \leq  \exp ( -\tau ) \cdot \big| u_1 (s) - u_2(s) \big|, \quad \tau \in (0.287, 1.386) .$$

For all $ u_1 (t ), u_2(t) \in X $ and $ \tau \in (0.287, 1.386) $, we have 
\begin{align*}
    \tau + \ln \big( D ~(Tu_1(t), Tu_2(t) ~\big) &  = \tau  + \ln~  [ \underset{ t \in [0, 1] }{ \sup }~ | Tu_1(t) - Tu_2(t) | + |Tu_1(0) - Tu_2(0) | ]   \\
    &  =  \tau + \ln ~ [ \underset{ t \in [0, 1] }{ \sup } ~ \big| \int_0^1 G (t, s)~ f (s, u_1(s))~ ds  -  \int_0^1 G (t,s)~ f (s, (u_2 (t) )~ ds  \big| \\
    & + \big | \int_0^1 G (0, t)~ f (s, u_1(s))~ ds  - \int_0^1~ f (s, u_2(s)) ~ ds \big| ]  \\
    &  \leq \tau + \ln ~ [ \underset{ t \in [0, 1] }{ \sup }~  \big | \int_0^1 g(t,s) ~ ds \big| \cdot \big| \int_0^1 \Big( \big( f (s, u_1(s)  \big) - f \big( s, u_2(s)  \big) \Big) ~ ds \big| \\
    & +  \underset{ t \in [0, 1] }{ \sup }~  \big | \int_0^1 g(0,s) ~ ds \big| \cdot \big| \int_0^1 \Big( \big( f (s, u_1(s)  \big) - f \big( s, u_2(s)  \big) \Big) ~ ds \big|]  \\
    & < \tau +  \ln ~ [ \frac{ 1 }{ 2 } \cdot \big|  \int_0^1 \big(  f (s, u_1(s) - f (s, u_2(s)) \big) ~ds \big| \\
    & + \frac{ 1 }{ 2 } \cdot \big|  \int_0^1 \big(  f (s, u_1(s) - f (s, u_2(s)) \big) ~ds \big|    ] \\
    &  < \tau + \ln ~ [ \big|  \int_0^1 \big(  f (s, u_1(s)) - f (s, u_2(s)) \big)~ ds \big|]  \\
    & = \tau + \ln ~ [ \underset{ s \in [0, 1] }{ \sup } ~ | f (s, u_1(s) - f (s, u_2(s))| ~ \int_0^1 ds ]  \\
    & \leq \tau + \ln ~ [ \exp (- \tau) \cdot \underset{ s \in [0, 1] }{ \sup } ~ |u_1(s) - u_2(s) | ] \\
    & = \tau - \tau + \ln ~ [ \underset{ s \in [0, 1] }{ \sup } ~ |u_1(s) - u_2(s) | ] \\
    & = \ln ~ [  \underset{ s \in [0, 1] }{ \sup } ~ |u_1(s) - u_2(s) | + |u_1(0) - u_2(0) | ] \\
    & = \ln ~ [ D ( u_1(t), u_2(t)) ].
\end{align*}
As a result, the conclusion is that the condition of Theorem  \ref{Perturbed fixed point theorem} are satisfied. Consequently, the Integral Equation  (\ref{Nu eq-2}) has unique solution. It can be easily checked that $ u (t ) = t $ is the exact solution of Equation (\ref{Nu eq-2}). \\

Now, we shall use the iteration method to underline the validity of our approaches
\begin{align*}
u_{ n + 1 } ( t) = T u_n(t) & = \int_0^1 G (t,s) \cdot f (s, u_n (s)) ~ ds  \\
& = \int_0^1 G (t,s) \cdot \frac{ (s + 0.5) }{ 2 } \cdot \sin u(s) ~ ds 
\end{align*}
For numerical iteration, we choose initial  guess $ u_0 (t ) = t $.

 \begin{figure}[h]
     \centering
     (a) \includegraphics[height=7cm, width=7cm]{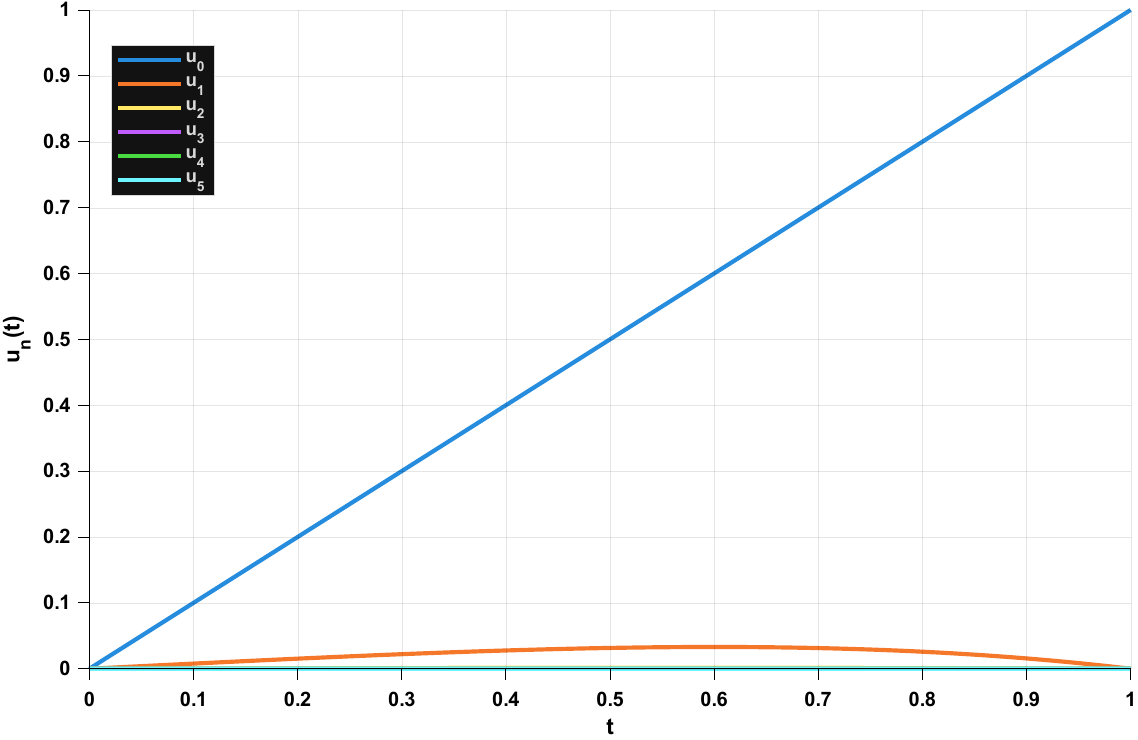}
     (b) \includegraphics[height=7cm, width=7cm]{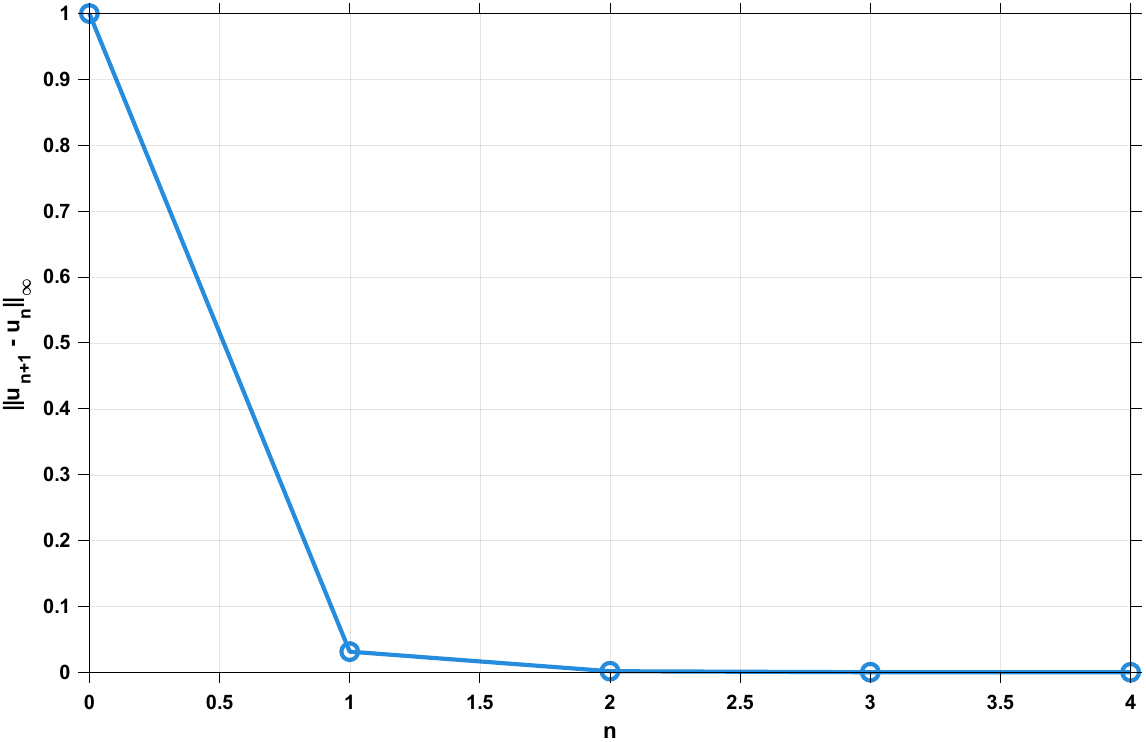}
     
     \caption{(a) Iteration $ u_n (t ) $  vs $ t $   ~~~     (b) $ ||u_{n +1} (t ) - u_n (t) | |_\infty$  vs $  n  $ }
     \label{Fig-3}
 \end{figure}
 Figures \ref{Fig-3} (a) show that the sequence $ u_{ n + 1} (t ) =T u_n (t) = \int_0^1 G(t, s) \cdot \frac{ (s + 0.5) }{ 2 } \cdot \sin u(s)~ ds $ converges to $ 0 $. 
Here, $ ||u_{n + 1 }(t) - u_n(t) ||_\infty = \max |u_{ n + 1}(t) -u_n|(t)|$ = sup-norm error of the iteration.

\section{Another generalized fixed-point theorem}
  In this section, we prove another theorem on the fixed points of an operator that includes the Banach fixed point theorem of a perturbed metric space \cite{P-metric}.  
              \begin{thm}\label{Next perturbed fixed point theorem}
             Let $ X $ be a complete perturbed metric space and let $ T $ be a mapping of $ X $ into itself. Suppose that for each positive integer $ n $, 
             $$ D ( T^nx, T^ny ) \leq a_n \cdot D (x, y) $$
             for all $ x, y \in X $ where $ a_n > 0 $ is independent of $ x, y $. If the series $ \sum_{ n = 1 }^{ \infty } a_n $ is convergent and $ T $ is a perturbed continuous function, then $ T $ has a unique fixed point in X.

             \begin{proof}
             	Let $ x_0 $ be an arbitrary element in $ X $ and consider the sequence $  \{ x_n \} $ of iterates 
             	$$ x_n = T^nx_0,  \quad  n = 1, 2, 3, \ldots. $$
             	We note that $ x_{  n +1 } = T^{ n + 1 }x_0 = T^n( Tx_0) = T^nx_1 $ and $ x_{ n + 1} = TT^nx_0 = Tx_n $.\\
             	Now,  we have
             	\begin{align*}
             		D(x_n, x_{ n + 1} ) & =  D( T^nx_0, T^{ n + 1}x_0 ) \\
             		& = D ( T^nx_0, T^nx_1) \\
             		& \leq a_n \cdot  D(x_0, x_1). 
             	\end{align*}
             Therefore
             \begin{equation}\label{ Next F.T equ 1}
             	D(x_n, x_{ n + 1 } ) \leq a_n  \cdot  D (x_0, x_1).
             \end{equation}
             If $x_0 = x_1 $ then a fixed point is obtained. Let therefore $ x_0 \ne x_1 $ and $ k $ be a positive integer with $ k > D (x_0, x_1) $. 
             \\
             As the series $ \sum_{ n = 1 }^{ \infty } a_n $ is convergent, then $ \underset{ n \to \infty }{ \lim } a_n = 0 $.\\
             From (\ref{ Next F.T equ 1}), we have 
             \begin{align*}
             &	D(x_n, x_{ n + 1} ) \leq a_n D(x_0, x_1) \\
             or~ ~ & d(x_n, x_{ n + 1} ) + P (x_n, x_{ n + 1} ) \leq a_n D (x_0, x_1)\\
             or~~ & d (x_n, x_{ n + 1} ) \leq a_n D (x_0, x_1) \\
             or~~ & \underset{ n \to \infty }{ \lim } d (x_n, x_{ n + 1} ) \leq D (x_0, x_1) \underset{ n \to \infty }{ \lim } a_n \\
             or~~ & \underset{ n \to \infty }{ \lim } d (x_n, x_{ n + 1}  ) \leq 0  \\
             or~~ & \underset{ n \to \infty }{ \lim } d (x_n, x_{ n + 1} ) = 0 
             \end{align*}
            
            Now we will show that $ \{ x_n \} $ is a Cauchy sequence.
            \begin{align*}
            &	d (x_n, x_{ n + 1 } ) \leq d (x_n, x_{ n + 1} ) + d (x_{n + 1}, x_{ n + 2} ) + \ldots + d (x_{ n + p - 1}, x_{ n + p} ) \\
             or~~ & \underset{ n \to \infty }{ \lim } d (x_n, x_{ n + p} ) \leq \underset{ n \to \infty }{ \lim } [ d (x_n, x_{ n + 1 } ) + d (x_{ n + 1}, x_ { n + 2} ) + \ldots + d ( x_{ n + p -1}, x_{ n + p} ) ] \\
             or~~ & \underset{ n \to \infty }{ \lim } d ( x_n, x_{ n + p } ) = 0. 
            \end{align*}
        Therefore, $ \{ x_n \} $ is a Cauchy sequence and the completeness of $ X $ implies that the existence of $ \xi \in X $ such that 
         $$ \underset{ n \to \infty }{ \lim } x_n = \xi. $$
         Since $ T $ is a perturbed continuous mapping, then 
         $$ \underset{ n \to \infty }{ \lim } d ( Tx_n, T \xi) = 0 ~~  i.e.,~ \underset{ n \to \infty }{ \lim } d (x_ { n + 1}, \xi ) = 0. $$
         If for $ n $ be a positive integer, then we obtain 
         
         \begin{align*}
          	d (\xi, T \xi ) &  \leq d ( \xi, x_{ n + 1} ) + d (x_{ n + 1}, T \xi ) \\ 
          	& = d ( \xi, x_{ n + 1} ) + d (Tx_n, T \xi ) \\
          	& = d (\xi, x_{ n + 1} ) + d (Tx_n, T \xi ) \\
          	& = \underset{ n \to \infty }{ \lim } d ( \xi, x_{ n + 1} ) + d ( Tx_n, T \xi )  \\
          	& \leq 0.
         \end{align*}
            
          Therefore,  $ d( T \xi, \xi ) \implies T \xi = \xi $.\\
          We now prove the uniqueness. If $ \eta  $ be a fixed point of $ T $ then for any positive integer,
          $$ \eta = T^n \eta \text{ and } \xi = T^n \xi.  $$
          so, 
          \begin{align*}
          	& D (\xi, \eta ) = D (T^n \xi, T^n \eta ) \leq a_n \cdot  D ( \xi, \eta ) \\
          	or~~ & ( a_n - 1) \cdot  D (\xi, \eta ) \geq 0 
          \end{align*}
           
           If $ \xi = \eta $ then the uniqueness is done. \\
           From Lemma \ref{lemma 1},
           if $ \xi \ne \eta $ then $ D (\xi, \eta ) > 0 $ so $ ( a_n - 1) \geq 0 ~ i.e., a_n \geq 1 ~ \forall ~ n  $. \\
           So $ a_n $  cannot tend to zero and this contradiction shows that $ \xi = \eta $ \\
           Thus $ T $ has a unique fixed point in $ X $.
             \end{proof}

            \end{thm}

            We now deduce Banach's fixed point theorem using perturbed metric \cite{P-metric} from  fixed-point Theorem \ref{Next perturbed fixed point theorem}.
              
              \begin{proof}
              	Since $ T $ is a Banach fixed point theorem using  perturbed mapping \cite{P-metric}, there exists $ 0 < \alpha < 1 $ such that 
              	
              	\begin{equation}\label{Next F.P equ 2}
              		D(Tx, Ty) \leq \alpha \cdot  D (x, y), \quad  \forall~ x, y \in X 
              	\end{equation}
              
              For $ x, y \in X $, we get from (\ref{Next F.P equ 2})
              
              \begin{align*}
              	 & D (T^2 x, T^2 y) \leq \alpha \cdot  D (Tx, Ty) \leq \alpha^2 \cdot  D (x, y) \\
              	 &  D (T^3x, T^3y) \leq \alpha^3 \cdot  D (x, y)
              \end{align*}
          
               and in general $ D (T^nx, T^ny) \leq \alpha^n \cdot  D (x, y) $ \\
               Since the series $ \sum_{ n = 1 }^{ \infty } \alpha^n $ is convergent, by (\ref{Next perturbed fixed point theorem}), $ T $ has a unique fixed point in $ X $. 
               
              \end{proof}

               We now present an example which supports Theorem \ref{Next perturbed fixed point theorem}.
               
               \begin{eg}
               	 Let $ D : [0, 1] \times [0, 1] \to [0, \infty) $ be the mapping defined by 
               	 $$ D (x, y) = | x - y | + (x - y )^2, \quad  \forall ~  x, y \in X.  $$
               	 Then $ D $ is a perturbed metric on $ [0, 1] $ with respect to the perturbed mapping $ P : [0, 1] \times [0, 1] \to [0, \infty) $ given by $ P (x, y) = (x - y)^2 $.\\
Let us define a self-mapping $ T : [0, 1] \to [0, 1] $ by 
               	 $$ Tx = \frac{ x }{ 3 }, \quad  x \in [0, 1]. $$
               	 For all $ x, y \in [0, 1] $, we have
               	 $$ D (T^nx, T^ny) = | \frac{ x }{ 3^n } - \frac{ y }{ 3^n }| + ( \frac{ x }{ 3^n } - \frac{ y }{ 3^n } ) = \frac{ 1 }{ 3^n } \cdot | x - y | + \frac{ 1 }{ 3^{2n} } ( x - y)^2.  $$
               	 We know that $ \frac{ 1 }{ 3^n }, \frac{ 1 }{ 3^{ 2n } } < 1 $.\\
               	 Now, 
               	 \begin{align*}
               	 	D (T^nx, T^ny) &  = \frac{ 1 }{ 3^n } \cdot | x - y | + \frac{ 1 }{ 3^{ 2n} } ( x - y)^2 \\
               	 	& < | x - y | + (x - y)^2.
               	 \end{align*}
                 Since $ a_n > 0 $ for all $ n \in \mathbb{ N } $ then 
                $$ D (T^nx, T^ny) < a_n \cdot ( | x - y | + (x - y)^2 ) = a_n \cdot D(x, y), \quad \forall ~ x, y \in [0, 1]. $$
                $ T $ is also a perturbed continuous mapping. Therefore, all the conditions of Theorem \ref{Next perturbed fixed point theorem} are satisfied. Hence $ T $ has a unique fixed point $ x = 0 $.

               \end{eg}

               \section*{Conclusion}
               In this framework, first we establish a fixed-point theorem for $ F $-perturbed mapping in a perturbed metric space and provide an application of second-order boundary value problem. Using an iteration method based on the fixed point approach, we find the approximate solution of the second-order boundary value problem. The numerical results have verified that the approach employed in the article is valid.  Next, another fixed point theorem establishes which includes the Banach fixed point theorem in perturbed metric spaces. In  future, researchers will be benefited to develop more results in perturbed metric space by using our results of this manuscript. 

\section*{Acknowledgment}
The authors are thankful to the Department of Mathematics, Siksha-Bhavana, Visva-Bharati, India.  The  author D.B. acknowledges financial support awarded by CSIR-UGC NET (DECEMBER-2022/JUNE-2023), University Grand Commission (UGC), New Delhi, India, through research  fellowship [Ref. No 231610065558] for carrying out research work leading to the preparation of this manuscript.


    \end{document}